\documentclass[a4paper,oneside,10pt,reqno,mathscr,final]{amsart}

\usepackage{amssymb}
\usepackage{euscript}
\usepackage[frame,cmtip,curve,arrow,matrix,line,graph]{xy}

\usepackage{showkeys}
\usepackage{version}
\usepackage{color}
\usepackage[dvipdfmx,
bookmarks=true,
colorlinks=true]{hyperref}

\newcommand{\bb} {\mathbb}
\newcommand{\frk}{\mathfrak}

\newcommand{\bbC}{\bb{C}}
\newcommand{\bbF}{\bb{F}}
\newcommand{\bbQ}{\bb{Q}}
\newcommand{\bbR}{\bb{R}}
\newcommand{\bbZ}{\bb{Z}}
\newcommand{\calF}{\mathcal{F}}
\newcommand{\calP}{\mathcal{P}}
\newcommand{\fgl}{\frk{gl}}
\newcommand{\fsl}{\frk{sl}}
\newcommand{\ff}{\frk{f}}
\newcommand{\fs}{\frk{s}}

\newcommand{\wh}{\widehat}

\newcommand{\ep}{\epsilon}
\newcommand{\ve}{\varepsilon}

\newcommand{\NS}{\text{NS}}

\newcommand{\pr}{\operatorname{pr}}
\newcommand{\re}{\operatorname{Re}}
\newcommand{\End}{\operatorname{End}}
\newcommand{\Vir}{\operatorname{Vir}}
\newcommand{\SVir}{\operatorname{SVir}}
\newcommand{\BF}{\operatorname{\mathcal{H}\mathcal{C}}}

\newcommand{\longto}{\longrightarrow}
\newcommand{\longsimto}{\xrightarrow{\ \sim \ }}
\newcommand{\simto}{\xrightarrow{\sim}}

\newcommand{\no}{\genfrac{}{}{0pt}{1}{\circ}{\circ}}
\newcommand{\ket}[1]{\left|#1\right>}
\newcommand{\bra}[1]{\left<#1\right|}
\newcommand{\spr}[3]{\left<#1 , #2\right>_{#3}}

\theoremstyle{definition}
 \newtheorem{thm}{Theorem}[section]
 \newtheorem{lem}[thm]{Lemma}
 \newtheorem{prop}[thm]{Proposition}
 \newtheorem{dfn}[thm]{Definition}
 \newtheorem{eg}[thm]{Example}
 \newtheorem{fct}[thm]{Fact}
 \newtheorem{rmk}[thm]{Remark}
 \newtheorem{cnj}[thm]{Conjecure}
\newtheorem*{thm*}{Theorem}

\newenvironment{ack}{\par {\bf Acknowledgement.} }{\par}

\numberwithin{equation}{section}
\allowdisplaybreaks

\begin{document}


\title{Singular vectors of $N=1$ super Virasoro algebra via Uglov symmetric functions}
\author{Shintarou Yanagida}
\address{Research Institute for Mathematical Sciences,
Kyoto University, Kyoto 606-8502, Japan}
\email{yanagida@kurims.kyoto-u.ac.HP}

\date{August 24, 2015}


\begin{abstract}
We give a proof of the conjecture by 
Belavin, Bershtein and Tarnopolsky,
which claims that 
a bosonization of singular vectors of $N=1$ super Virasoro algebra 
can be identified with Uglov (Jack $\fgl_p$) symmetric functions with $p=2$.
\end{abstract}

\maketitle

\section{Introduction}
\label{sect:intro}

\subsection{}

Singular vectors of representations of Lie algebras 
(or quantum groups, $W$ algebras and so on)
are important from some viewpoints.
The first viewpoint is that the existence of a singular vector 
implies the reducibility of the module.
The second one, which is the motivation of this note,
is that in several cases the singular vector can be expressed 
explicitly in terms of special functions.

One of the most surprising examples is the theorem of 
Mimachi and Yamada \cite{MY} 
that singular vectors $\chi_{r,s}$ of Verma modules of the Virasoro Lie algebra 
$\Vir$
can be identified with Jack symmetric functions $P_{(r^s)}(\beta)$,
where the parameter $\beta$ is related to the central charge $c$ of $\Vir$ 
as $c=1-8(\beta-1)^2/\beta$.

Let us also mention the work of Shiraishi, Kubo, Awata and Odake \cite{SKAO} 
that the singular vectors of Verma modules of the deformed Virasoro algebra 
$\Vir_{q,t}$ 
can be identified with Macdonald symmetric functions $P_{(r^s)}(q,t)$,
where the parameters $(q,t)$ of the functions coincide with those of the algebra.
This work can be considered as a natural $q$-analogue of 
the work of Mimachi and Yamada.

Recently Belavin, Bershtein and Tarnopolsky \cite[\S 4]{BBT} established 
a super-analogue of this Virasoro-Jack relation.
They claim that 
singular vectors of Verma module of the $N=1$ super Virasoro algebra 
$\SVir$ are, by the bosonization, 
identified with the Uglov symmetric function 
$P^{(\gamma,2)}_{(r^s)}$.
The purpose of this note is to prove this conjecture.

\subsection{}

Now let us state the claim precisely.
Denote by $\SVir$ the Neveu-Schwartz sector of the $N=1$ super Virasoro algebra 
(see Definition \ref{dfn:SVir} for the detail).
Let $M(c,h)$ be the Verma module of $\SVir$, 
where $c$ is the central charge and $h$ is the highest weight
(see Definition \ref{dfn:Verma}).
It is known that $M(c,h)$ has a singular vector 
(see Definition \ref{dfn:singular})
if and only if $h=h^{\NS}_{r,s}$ for a set of integers $(r,s)$ with 
$r-s \equiv 0 \pmod{2}$
(see Fact \ref{fct:singular} for the definition of $h^{\NS}_{r,s}$
 and the detail).

Let us also denote by $\calF(\alpha)$ the Fock representation of $\SVir$
with highest weight $\alpha$.
There is an intertwiner $M(c,h) \to \calF(\alpha)$
with 
$$
 h=\alpha^2-\rho \alpha,\quad
 c=3/2-12\rho^2.
$$
(see Fact \ref{fct:bosonization}).
Under fixed $\rho$, there are two $\alpha$'s corresponding to $h$.
Let us denote by $\alpha_{r,s}^{\pm}$ the corresponding 
highest weight to $h_{r,s}$.
Let us denote by $\chi_{r,s} \in M(c,h_{r,s})$ the singular vector,
and  by $\chi_{r,s}^+ \in \calF(\alpha_{r,s}^+)$ 
the image under the intertwiner $M(c,h) \to \calF(\alpha)$.
(see Fact \ref{fct:singular} for the explicit definition of $\chi_{r,s}^+$).

For each $t \in \bbC$,
we have an isomorphism $\fs_t:\calF(\alpha_{r,s}) \to \Lambda_{\bbC}$
of vector spaces,
where $\Lambda_{\bbC}$ is the space of symmetric functions 
with coefficients in $\bbC$
(see Lemma \ref{lem:sf}).

\begin{thm*}
Set $t^+ := \rho + \sqrt{\rho^2+1}$.
Then 
$$
 \fs_{t^+}(\chi_{r,s}^+) \propto P^{(1/t_+^2,2)}_{(r^s)}.
$$
where the right hand side is the Uglov symmetric function 
(see Definition \ref{dfn:Uglov})
\end{thm*}


\subsection{}

Now let us explain the organization of this note briefly.
In \S \ref{sect:SVir} we review the $N=1$ super Virasoro algebra, 
and singular vectors of the Verma module. 
The Uglov symmetric function is also reviewed 
and the conjecture of \cite{BBT} will be explained 
with some examples.

In \S \ref{sect:main} we give a proof of the conjecture.
We will show that the singular vectors are the eigenfunctions
of a certain operator $C^1_0(1/t_+^2)$ (Proposition \ref{prop:C10}).
We will also see that this property and the triangular expansion 
with respect to the dominance ordering 
characterises the Uglov symmetric function uniquely 
(Proposition \ref{prop:Uglov}).

In the appendix we give remarks on the following topics.
The first one given in \S \ref{sect:Selberg} 
is that the existence of singular vectors
is guaranteed by the non-triviality of the Selberg integral \cite{S}.
Note that in the (non-super) Virasoro case, 
it is already stated and proven by \cite{TK}.
The second on in \S \ref{} is that 
our operator $C^1_0(\gamma)$ is a limit of 
a difference-differential operator $C^1_N(\gamma)$ 
acting on the symmetric polynomials of $N$-variables,
where the word `limit' means the limit of $N \to \infty$.

\subsection{Notation}
All the objects are defined over the complex number field $\bbC$ 
unless otherwise stated.

We follow \cite{M} for the notations of partitions and related notions.
%
A partition is a finite sequence of positive integers 
ordered by size.
We include the empty sequence $\emptyset$ as a partition.
For a partition $\lambda = (\lambda_1,\ldots,\lambda_k)$,
its length and total sum are denoted as 
$$
 \ell(\lambda):=k,\quad |\lambda| := \sum_{i=1}^k \lambda_i.
$$
We also set 
$$
 z_\lambda := \prod_{i \in \bbZ_>0} i^{m_\lambda(i)} m_\lambda(i)!,\quad
 m_\lambda(i) := \#\{ j \in \bbZ_{>0} \mid m_j = i \}.
$$
The set of all partitions will be denoted by $\calP$.


We also follow \cite{M} for the notation of symmetric functions.
$\Lambda$ denotes the space of symmetric functions over $\bbZ$.
For a ring $R$, we denote the coefficient extension by 
$\Lambda_R := \Lambda \otimes_{\bbZ} R$.

Some bases of $\Lambda$ will be used repeatedly.
$m_\lambda$ denotes the monomial symmetric function for a partition $\lambda$.
$e_n$ denotes the $n$-th elementary symmetric function
and $e_\lambda := e_{\lambda_1} \cdots e_{\lambda_k}$. 
The sets $\left\{ m_\lambda \mid \lambda \in \calP \right\}$ and 
$\left\{ e_\lambda \mid \lambda \in \calP \right\}$ are bases of $\Lambda$. 
$p_n$ denotes the $n$-th power-sum symmetric function,
and $p_\lambda := p_{\lambda_1} \cdots p_{\lambda_k}$ for a partition $\lambda$.
The set $\left\{ p_\lambda \mid \lambda \in \calP \right\}$ is a basis of 
$\Lambda_{\bbQ}$.

For two (commuting) indeterminates $q$ and $t$, 
the monic Macdonald symmetric function 
is denoted by $P_\lambda(q,t)$.
It is a member of $\Lambda_{\bbF}$ with $\bbF := \bbQ(q,t)$.
The set $\left\{P_\lambda(q,t) \mid \lambda \in \calP \right\}$ 
is a basis of $\Lambda_{\bbF}$,
characterized by the following properties.
\begin{itemize}
\item 
It has the expansion 
$P_\lambda(q,t) = m_{\lambda} + \sum_{\mu < \lambda} c_{\lambda,\mu} m_\mu$
with respect to the dominance ordering $<$ among partitions. 

\item
$$
 \spr{P_\lambda(q,t)}{P_\mu(q,t)}{q,t} \propto \delta_{\lambda,\mu},
$$
where $\spr{-}{-}{q,t}$ is the inner product on $\Lambda_{\bbF}$ 
defined by
$$
 \spr{p_\lambda}{p_\mu}{q,t} := 
  \delta_{\lambda,\mu} z_\lambda
  \prod_{i=1}^{\ell(\lambda)}\dfrac{1-q^{|\lambda_i|}}{1-t^{|\lambda_i|}}
$$ 
\end{itemize}


\begin{ack}
The author is supported by the Grant-in-aid for 
Scientific Research (No.\ 25800014), JSPS.
This work is also partially supported by the 
JSPS for Advancing Strategic International Networks to 
Accelerate the Circulation of Talented Researchers
``Mathematical Science of Symmetry, Topology and Moduli, 
  Evolution of International Research Network based on OCAMI"''.

The author would like to thank R.~Kodera and H.~Awata
for helpful discussions.
He would like to thank K.~Wada for 
the invitation to Shinshu University in July 2015, 
where a seminar talk is given on the content of this note.

This note is finished during the author's stay at UC Davis. 
The author would like to thank the institute for support and hospitality.
\end{ack}

\section{Singular vectors of $N=1$ super Virasoro algebra}
\label{sect:SVir}

Our notation basically follows \cite[\S III]{KM}.

\subsection{$N=1$ super Virasoro algebra}

In this paper we only consider the Neveu-Schwarz sector.
\begin{dfn}\label{dfn:SVir}
We denote by 
$\SVir = \SVir_{N=1}^{\NS}$
the Neveu-Schwarz sector of the $N=1$ super Virasoro algebra.
It is the Lie superalgebra 
generated by $L_n$ ($n \in \bbZ$), $G_k$ ($k \in \bbZ + 1/2$) 
and central $C$ with 
\begin{align*}
&[L_m,L_n] = (m-n)L_{m+n}+\dfrac{1}{12}\left(m^3-m\right)\delta_{m+n,0} \, C,\\ 
&[L_n,G_k] = \left(\dfrac{n}{2}-k\right)G_{n+k},\\ 
&[G_k,G_l]_+ = 2L_{k+l} + \dfrac{1}{3}\left(k^2-\dfrac{1}{4}\right)C.
\end{align*}
\end{dfn}

\begin{rmk}\label{rmk:SVir:current}
\begin{enumerate}
\item 
$G_k$ ($k \in \bbZ + 1/2$) for Neveu-Schwarz algebra, 
$G_k$ ($k \in \bbZ$) Ramond algebra.

\item
We shall use the following current form.
\begin{align*}
T(z) := \sum_{n \in \bbZ}L_n z^{-n-2},\quad
G(z) := \sum_{k \in \bbZ+1/2}G_k z^{-k-3/2}.
\end{align*}
Then the defining relation is rewritten as the following OPE.
\begin{align*}
&T(z) T(w) \sim 
 \dfrac{C/2}{(z-w)^4}+\dfrac{2}{(z-w)^2}T(w)+\dfrac{1}{z-w}(\partial T)(w),
\\
&T(z) G(w) \sim 
 \dfrac{3/2}{(z-w)^2}G(w)+\dfrac{1}{z-w}(\partial G)(w),
\\
&G(z) G(w) \sim 
 \dfrac{2C/3}{(z-w)^3}+\dfrac{2}{z-w}T(w).
\end{align*}
Here $\sim$ is the equality modulo regular parts 
and $(\partial T)(z) := \partial_z T(z)$, 
$(\partial G)(z) := \partial_z G(z)$.

\end{enumerate}
\end{rmk}

\begin{dfn}\label{dfn:Verma}
\begin{enumerate}
\item 
For $c,h \in \bbC$,
the Verma module  $M(c,h)$ 
is the $\SVir$-module generated by 
the highest weight vector $\ket{c,h}$ with 
$$
 L_n \ket{c,h} = 0, \ (n>0) \quad
 L_0 \ket{c,h} = h \ket{c,h},\quad
 G_k \ket{c,h} = 0, \ (k>0) \quad
 C   \ket{c,h} = c\ket{c,h}.
$$

\item
For $c,h \in \bbC$,
the dual Verma module  $M^*(c,h)$ 
is the right $\SVir$-module generated by 
the highest weight vector $\bra{c,h}$ with 
$$
 \bra{c,h} L_{-n} = 0 \ (n>0), \quad
 \bra{c,h} L_0    = h \bra{c,h},\quad
 \bra{c,h} G_{-k} = 0 \ (k>0), \quad
 \bra{c,h} C      = c \bra{c,h}.
$$

\item
The contravariant form  
$\cdot: M^*(c,h) \times M(c,h) \to \bbC$ is the bilinear form 
uniquely characterized by the properties
$$
 \bra{c,h} \cdot \ket{c,h} = 1,\quad
 w x \cdot v = w \cdot x v \ 
 \left(\forall \, v \in M(c,h),\ w \in M^*(c,h),\ 
  x \in U(\SVir)\right).
$$
Hereafter we omit the symbol $\cdot$  
since it doesn't make confusion by the second property.
\end{enumerate}
\end{dfn}

\begin{rmk}
\begin{enumerate}
\item 
$M(c,h)$ has a natural PBW-type basis 
$$
 \left\{L_{-\lambda^a} G_{-\lambda^b} \ket{c,h} 
   \, \big| \,  
   \text{$\lambda = (\lambda^a,\lambda^b)$: super-partitions}
 \right\}.
$$
Here a super-partition $\lambda = (\lambda^a,\lambda^b)$ 
consists of the usual partition
$$
 \lambda^a = (\lambda^a_1 \ge \lambda^a_2 \ge \cdots \ge \lambda^a_l) 
 \subset \bbZ_{\ge 1}
$$ 
and the half-integer valued strict partition
$$
 \lambda^b = (\lambda^b_1 > \lambda^b_2 > \cdots > \lambda^b_m) 
 \subset \dfrac{1}{2}+\bbZ_{\ge 0}.
$$
For $\lambda = (\lambda^a,\lambda^b)$  
we used the abbreviation
$$
 L_{-\lambda^a} G_{-\lambda^b} := 
 L_{-\lambda^a_l} \cdots L_{-\lambda^a_2} \, L_{-\lambda^a_1} \, 
 G_{-\lambda^b_m} \cdots G_{-\lambda^b_2} \, G_{-\lambda^b_1}.
$$
Hereafter we will denote by $\calP^{\NS}$ the set of super-partitions.

\item
By the defining definition, 
$M(c,h)$ is $L_0$-graded as
$$
 M(c,h) = \bigoplus_{2n \in \bbZ_{\ge0}} M(c,h)_n,\quad
 M(c,h)_n := \{ v \in M(c,h) \mid L_0 v = (n+h)v \}.
$$
Each graded component $M(c,h)_n$ has the PBS-basis
$$
 \left\{L_{-\lambda^a} G_{-\lambda^b} \ket{c,h} 
   \, \big| \,  |\lambda| = n
 \right\},
$$
where for a super-partition $\lambda = (\lambda^a,\lambda^b)$ 
we defined
$$
 |\lambda| := |\lambda^a| + |\lambda^b| 
            = \sum_{i=1}^l \lambda^a_i + \sum_{j=1}^m \lambda^b_j.
$$
Thus the dimension of the graded component 
is given by 
$$
 \dim M(c,h)_n  = p_{\NS}(n) := \#\{\lambda \in \calP^{\NS} \mid |\lambda| =n\}
$$
and its generating function is given by 
$$
 \sum_{2n \in \bbZ_{\ge0}} x^n p_{\NS}(n) 
 = \dfrac{\prod_{k \in \bbZ_{\ge0}+1/2}(1+x^k)}{\prod_{m \in \bbZ_{>0}}(1-x^m)}.
$$

\item
Similar statements hold for $M^*(c,h)$.
The PBW-type basis is written as 
$$
 \left\{L_{-\lambda^a} G_{-\lambda^b} \ket{c,h} 
   \, \big| \,  
   \lambda = (\lambda^a,\lambda^b) \in \calP^{\NS}
 \right\}.
$$
with the abbreviation
$$
 G_{\lambda^b}  L_{\lambda^a} := 
 G_{\lambda^b_1} \, G_{\lambda^b_2} \cdots G_{\lambda^b_m}
 L_{\lambda^a_1} \, L_{\lambda^a_2} \cdots L_{\lambda^a_l}.
$$

\item
For $w \in M^*(c,h)_m$ and $v \in M(c,h)_n$, we have 
$$
 w \cdot v = 0 \ \text{ if } \ m \neq n.
$$
\end{enumerate}
\end{rmk}

The irreducibility of the Verma module $M(c,h)$ is 
encoded in the determinant of the contravariant form.
As in the (non-super) Virasoro algebra,
the determinant has a factorized formula, 
which was conjectured by Kac \cite{K}.

\begin{fct}[{conjectured in \cite{K}}]
\label{fct:Kac}
For $n$ satisfying $2n \in \bbZ_{\ge 0}$,
set the $p_{\NS}(n) \times p_{\NS}(n)$ matrix $K_n$ by 
$$
 K^{\NS}_n := 
 \left(\bra{c,h} G_{\lambda^b}  L_{\lambda^a}
  L_{-\mu^a} G_{-\mu^b} \ket{c,h}\right)_{|\lambda|=|\mu|= n}.
$$
Then its determinant has the form 
\begin{align}
\label{eq:Kac}
\det\left(K^{\NS}_n\right) \propto 
 \prod_{\substack{r,s \in \bbZ_{>0},\, r s \le 2 n, \\ r \equiv s \pmod{2}}}
 \left(h-h_{r,s}^{\NS}\right)^{p_{\NS}(n-rs/2)}
\end{align}
where the zeros of the determinant are
\begin{align}\label{eq:hrsNS}
 h_{r,s}^{\NS} := 
 \dfrac{1}{8}\left(r t_+ + s t_-\right)^2 - \dfrac{\rho^2}{2}, \quad
\end{align}
with $\rho$ and $t_{\pm}$ defined by 
\begin{align}
\label{eq:tpm}
 c= \dfrac{3}{2}-12\rho^2,\quad
 t_{\pm} := \rho \pm \sqrt{\rho^2+1}. 
\end{align}
We obviously have $t_- =-  t_+^{-1}$.
\end{fct}

\begin{eg}
For $n \le 3/2$ we have
\begin{align*}
&K^{\NS}_{1/2} = \left(\langle G_{1/2} G_{-1/2} \rangle\right) = (2h),\quad
 K^{\NS}_{1}   = \left(\langle L_{1}   L_{-1}   \rangle\right) = (2h),
\\
&K^{\NS}_{3/2}=
 \begin{pmatrix} 
  \langle G_{1/2} L_{1} L_{-1} G_{-1/2} \rangle & \langle G_{1/2} L_{1} G_{-3/2} \rangle \\
  \langle G_{3/2}       L_{-1} G_{-1/2} \rangle & \langle G_{3/2}       G_{-3/2} \rangle
 \end{pmatrix}
= \begin{pmatrix} 4h^2+4h & 4h \\ 4h & 2h + 2c/3 \end{pmatrix}
\end{align*}
and
\begin{align*}
\det\left(K^{\NS}_1\right) = 2 \left(h-h^{\NS}_{1,1}\right),\quad
\det\left(K^{\NS}_2\right) = 2 \left(h-h^{\NS}_{1,1}\right),\quad
\det\left(K^{\NS}_3\right) = 8 \left(h-h^{\NS}_{1,1}\right)
                               \left(h-h^{\NS}_{1,3}\right)\left(h-h^{\NS}_{3,1}\right).
\end{align*}
\end{eg}

In the next subsection we follow 
the strategy of \cite{KM} proving Fact \ref{fct:Kac}
using the bosonization of the algebra.
It is a super-analogue of \cite{FF},
where the Virasoro algebra case was shown,

Let us also mention that Iohara and Koga \cite{IK1} 
obtained the Bernstein-Gelfand-Gelfand type resolutions of 
the Verma module of $\SVir$,
and as a corollary Fact \ref{fct:Kac} follows.
See also \cite{IK2} for the study of the Fock module.

\subsection{Bosonization}

\begin{dfn}
Heisenberg-Clifford algebra $\BF$ 
is the Lie superalgebra generated by $a_n$ ($n\in\bbZ$) 
and $b_k$ ($k\in\bbZ+1/2$) and $1$ with 
\begin{align*}
[a_m,a_n] = m \delta_{m+n,0},\quad
[b_k,b_l]_+ = \delta_{k+l,0},\quad
[a_n,b_k] = 0. 
\end{align*}
The normal ordering will be denoted by $\no \cdots \no$.
\end{dfn}

\begin{rmk}
As in Remark \ref{rmk:SVir:current},
we can rewrite the defining relation of $\BF$ 
via the currents
$$
 a(z) = \sum_{n \in \bbZ} a_n z^{-n-1},\quad
 b(z) = \sum_{k \in \bbZ+1/2} b_k z^{-k-1/2}
$$
as 
\begin{align}\label{eq:HA:current}
 a(z) a(w) = \dfrac{1}{(z-w)^2} + \no a(z) a(w) \no,\quad
 b(z) b(w) = \dfrac{1}{z-w} + \no b(z) b(w) \no.
\end{align}
\end{rmk}

\begin{dfn}
For $\alpha \in \bbC$,
the Fock module $\calF(\alpha)$ is the $\BF$-module 
generated by the highest weight vector $\ket{\alpha}$ with 
\begin{align}\label{eq:BF:rel}
 a_n \ket{\alpha} = 0 \ (n > 0),\quad
 a_0 \ket{\alpha} = \alpha \ket{\alpha} \quad
 b_k \ket{\alpha} = 0 \ (k > 0).
\end{align}
\end{dfn}

\begin{fct}\label{fct:bosonization}
\begin{enumerate}
\item 
For $\rho \in \bbC$, the map 
\begin{align}
\label{eq:bosonization}
\begin{array}{lll}
\displaystyle
T(z) := \sum_{n\in\bbZ} L_n z^{-n-2}
&\longmapsto 
&\displaystyle
  \dfrac{1}{2} \no a(z)^2 \no + \rho \, (\partial a)(z) + 
  \dfrac{1}{2} \no (\partial b)(z) \, b(z) \no 
\\
\displaystyle
G(z) := \sum_{k \in \bbZ+1/2} G_k z^{-k-3/2} 
&\longmapsto 
&\displaystyle
 b(z) \, a(z) + 2 \rho \, (\partial b)(z),
\\
\displaystyle
C
&\longmapsto 
&c = \dfrac{3}{2}-12\rho^2,
\end{array}
\end{align}
gives an injective algebra homomorphism 
$$
 \ff=\ff_{\rho}: U(\SVir) \longto \wh{U}(\BF).
$$
Here $\wh{U}(\BF)$ is a certain completion of 
the enveloping algebra $U(\BF)$ of $\BF$. 

\item
The linear map 
\begin{align}\label{eq:bosonization:MF}
 M(c,h) \longto \calF(\alpha)
\end{align}
induced by  $\ket{c,h} \longmapsto \ket{\alpha}$
and $\ff_{\rho}$ in (1) 
defines a homomorphism of $\SVir$-modules if 
\begin{align}\label{eq:h-alpha}
 h=\dfrac{\alpha^2}{2}-\rho \alpha.
\end{align}
\end{enumerate}
\end{fct}

Equivalently, the map $\ff$ is given by
\begin{align*}
&\ff(L_n) = \dfrac{1}{2} \sum_{m \in \bbZ} \no a_{n} a_{m-n} \no
            - \rho (n+1) a_n
            - \dfrac{1}{2} \sum_{k \in \bbZ+1/2} 
              \left(k+\dfrac{1}{2}\right) \no b_{k} b_{n-k} \no,
\\
&\ff(G_k) = \sum_{m \in \bbZ} b_{k-m} a_{m} -2\rho\left(k+\dfrac{1}{2}\right)b_k.
\end{align*}

\begin{proof}
(1) can be shown by using \eqref{eq:BF:rel} 
and Remark \ref{rmk:SVir:current} (2).
(2) is an immediate corollary of (1).
The condition \eqref{eq:h-alpha} is the consequence of 
$\ff(L_0)\ket{\alpha} = \left(a_0^2/2 - \rho a_0\right)\ket{\alpha}$.
\end{proof}

\subsection{Singular vectors via screening operators}

The proof of the Kac determinant formula \eqref{eq:Kac}
is reduced to the construction of the singular vectors 
via bosonization.

\begin{dfn}\label{dfn:singular}
Let $M$ be a $\SVir$-module.
An element $\chi \in M$ is called a singular vector of $M$ if
$$
 L_n \chi = 0\ (n>0),\quad 
 G_k \chi = 0\ (k>0).
$$
\end{dfn}

\begin{eg}
The Verma module $M(c,h)$ with 
$$
 c= \dfrac{3}{2}-12\rho^2,\quad
 h^{\NS}_{r,s} = \dfrac{1}{8}\left(r t_{+} + s t_{-} \right)^2 - \dfrac{\rho^2}{2}
$$ 
given in Fact \ref{eq:Kac} 
has a singular vector $\chi_{r,s}$.
For $r s \le 3$ it is given by 
\begin{align*}
&\chi_{1,1} = G_{-1/2}\ket{c,h^{\NS}_{1,1}},\\
&\chi_{3,1} = 
 \left(L_{-1}G_{-1/2}-t^{ 2}_{+}G_{-3/2} \right)\ket{c,h^{\NS}_{3,1}},\\
&\chi_{1,3} = 
 \left(L_{-1}G_{-1/2}-t^{2}_{-}G_{-3/2} \right)\ket{c,h^{\NS}_{1,3}}.
\end{align*}
\end{eg}

Hereafter we identify elements of $\SVir$ as those of $\wh{U}(\BF)$ 
under the bosonization map $\ff$ in Fact \ref{fct:bosonization}
and write $L_n = \ff(L_n)$ etc., suppressing the symbol $\ff$.
Likewise we may consider elements of $M(c,h)$ as those of $\calF(\alpha)$ 
mapped by the intertwiner \eqref{eq:bosonization:MF}.

\begin{eg}\label{eg:chi:Fock}
Let us consider the image of $\chi_{r,s}$ under \eqref{eq:bosonization:MF}.
By the definition \eqref{eq:hrsNS} of $h^{\NS}_{r,s}$ and 
the correspondence \eqref{eq:h-alpha} of highest weights,
we have the intertwiner 
$$
 M(c,h_{r,s}) \longto \calF(\alpha_{r,s}),\quad
 \alpha_{r,s} := \dfrac{r+1}{2} t_{+} + \dfrac{s+1}{2} t_{-}.
$$
Under this map we have 
\begin{align*}
G_{-1/2}\ket{c,h_{r,s}} & \longmapsto 
  \alpha_{r,s} b_{-1/2}\ket{\alpha_{r,s}},\\
G_{-3/2}\ket{c,h_{r,s}} & \longmapsto  
  \left(a_{-1}b_{-1/2}+(2\rho+\alpha_{r,s})b_{-3/2}\right)\ket{\alpha_{r,s}},
\\
L_{-1}G_{-1/2}\ket{c,h_{r,s}} & \longmapsto 
 \left(\alpha_{r,s}^2 a_{-1}b_{-1/2}+ \alpha_{r,s} b_{-3/2}\right)\ket{\alpha_{r,s}}
\end{align*}
so that, using $t_{\pm}^2-2\rho t_{\pm} =1$, we find 
\begin{align*}
\chi_{1,1} \longmapsto 
& \alpha_{1,1}b_{-1/2} \ket{\alpha_{1,1}}
 = \left(t_{+}+t_{-}\right) b_{-1/2} \ket{\alpha_{1,1}},
\\
\chi_{3,1} \longmapsto 
& \left(\left(\alpha_{3,1}^2-t_{+}^2\right) a_{-1}b_{-1/2} 
       +\left(\alpha_{3,1}(1-t_{+}^2)-2\rho t_+^{2}\right)b_{-3/2}\right)
  \ket{\alpha_{3,1}}\\
&
  = \left(t_{+}+t_{-}\right) \left(3t_{+}+t_{-}\right)
    \left(a_{-1}b_{-1/2} -t_{+} b_{-3/2}\right) \ket{\alpha_{3,1}}
\\
\chi_{1,3} \longmapsto  
&
 \left(\left(\alpha_{1,3}^2-t_{-}^2\right) a_{-1}b_{-1/2} 
        +\left(\alpha_{1,3}(1-t_{-}^2)-2\rho t_{-}^2\right)b_{-3/2}\right)
   \ket{\alpha_{1,3}}\\
&
  = \left(t_{+}+t_{-}\right) \left(t_{+}+3t_{-}\right)
    \left(a_{-1}b_{-1/2} - t_{-} b_{-3/2}\right) \ket{\alpha_{1,3}}.
\end{align*}
\end{eg}

One can construct singular vectors $\chi_{r,s}$ realized in  $\calF(\alpha)$
by the so-called screening operators.
To write down the screening operators, let us introduce

\begin{dfn}
Let $\wh{\BF}$ be the extended Lie superalgebra of $\BF$ by the element 
$\pi_0$ such that $[a_0,\pi_0]=1$.
\end{dfn}

\begin{rmk}
For $t \in \bbC$, the expression 
$$
 e^{t \pi_0} := \sum_{n=0}^\infty \dfrac{1}{n!} t^n \pi_0^n
$$ 
can be considered as an intertwiner $\calF(\alpha) \to \calF(\alpha+t)$ 
of $\BF$-modules since
\begin{align}\label{eq:etp-intertwiner}
a_0 e^{t \pi_0} \ket{\alpha} = (t+\alpha) e^{t \pi_0} \ket{\alpha}
\end{align}
so we may identify $ e^{t \pi_0} \ket{\alpha}$ and $\ket{\alpha+t}$.
\end{rmk}

\begin{fct}[\cite{KM}]
For $t \in \bbC$, the element 
\begin{align*}
W_t(z) := t b(z) V_t(z)
\end{align*}
of $\wh{U}(\wh{\BF})$  with  
\begin{align*}
V_t(z) := 
         \exp\left( t\sum_{n>0} \dfrac{1}{n} a_{-n} z^n\right) 
         \exp\left(-t\sum_{n>0} \dfrac{1}{n} a_{n} z^{-n}\right) 
         e^{t \pi_0} z^{t a_0}
\end{align*}
satisfies the following relations.
\begin{align*}
&T(z) W_t (w) \sim 
 \dfrac{1}{(z-w)^2}\left(\dfrac{1}{2}t^2-\rho t + \dfrac{1}{2}\right) W_t(w)
 + \dfrac{1}{z-w}(\partial W_t)(w),\\
&G(z) W_t (w) \sim 
 \dfrac{1}{(z-w)^2} \left(t^2-2\rho t\right) V_t(w)
 + \dfrac{1}{z-w} (\partial V_t)(w).
\end{align*}
\end{fct}

\begin{proof}
The result follows from the usual calculus using 
$$
V_t(z) = \no e^{t \varphi(z)} \no, \quad 
\varphi(z) := \pi_0 + a_0 \log z - \sum_{n \neq 0}\dfrac{1}{n}a_n z^{-n}.
$$
For the completeness we sketch the outline.
We first note
$$
 (\partial \varphi)(z)=a(z),\quad 
 \varphi(z) \, \varphi(w) \sim \log(z-w).
$$
Then we immediately have
\begin{align*}
&a(z) V_t(w) 
  \sim \dfrac{t}{z-w} V_t(w),\quad
 (\partial a)(z) V_t(w) 
  \sim \dfrac{-t}{(z-w)^2} V_t(w),\\
&\no a(z)^2 \no \, V_t(w) 
  \sim \dfrac{t^2}{(z-w)^2} V_t(w)
       + \dfrac{2t}{z-w}\no a(w) V_t(w) \no .
\end{align*}
Recalling the bosonization map \eqref{eq:bosonization} 
in Fact \ref{fct:bosonization}, we have the standard result 
\begin{align*}
T(z) V_t(w) \sim 
 \dfrac{1}{(z-w)^2} \left(\dfrac{1}{2}t^2-\rho t\right) V_t(w) 
 + \dfrac{1}{z-w} (\partial V_t)(w).
\end{align*}
Now the first relation follows easily by this result and
$$
 \no (\partial b)(z) b(z) \no \, b(w) \sim 
 \dfrac{2}{z-w}(\partial b)(w) + \dfrac{1}{(z-w)^2} b(w).
$$
The second relation can be similarly obtained.
\end{proof}


As a corollary, 
if $t^2-2\rho t =1$, namely if $t = t_{\pm}$ defined at \eqref{eq:tpm},
then
\begin{align*}
&T(z) W_{t_{\pm}} (w) \sim 
 \dfrac{\partial}{\partial z} \left(\dfrac{1}{z-w} W_{t_{\pm}}(w)\right),\quad
 G(z) W_{t_{\pm}} (w) \sim 
 \dfrac{\partial}{\partial z} \left(\dfrac{1}{z-w} V_{t_{\pm}}(w)\right).
\end{align*}
Thus symbolically we have 
\begin{align*}
 \left[T(z), \oint dw W_{t_{\pm}}(w) \right] = 0,\quad
 \left[G(z), \oint dw W_{t_{\pm}}(w) \right] = 0.
\end{align*}
It implies that $\oint dz W_{t_{\pm}}(z)$
gives rise singular vectors of the Fock module 
(viewed as an $\SVir$-module under the map $\ff$).
$\oint dz W_{t_{\pm}}(z)$ is  called the screening operator.

\begin{fct}[\cite{KM}]\label{fct:singular}
For $\ep := \pm$ and $r,s \in \bbZ_{>0}$ with $r \equiv s \pmod{2}$,
there is an integration path $C$ such that the element
\begin{align}\label{eq:chi_rs}
 \chi^{\ep}_{r,s} := 
  \oint_C d z_1 \cdots d z_r 
   W_{t_{\ep}}(z_1) \cdots W_{t_{\ep}}(z_r) \ket{\alpha_{r,s}^\ep - rt_{\ep}}
\end{align}
is a non-zero singular vector of $\SVir$-module $\calF(\alpha_{r,s}^\ep)$ with 
$$
 \alpha_{r,s}^{\pm} := \dfrac{r+1}{2} t_{\pm} + \dfrac{s+1}{2} t_{\mp}.
$$
Here we choose $\rho$ appearing in the bosonization map 
$\ff:U(\SVir)\to\wh{U}(\BF)$ to be
$$
 \rho = \dfrac{t_{\ep}}{2} - \dfrac{1}{2 t_{\ep}} = \alpha_{0,0}^{\ep}.
$$
\end{fct}

Since $W_t(z)$ contains the intertwiner $e^{t \pi_0}$,
the element $\chi^{\ep}_{r,s}$ is an element of $\calF(\alpha^{\ep}_{r,s})$ 
by \eqref{eq:etp-intertwiner}. 
The proof of the Kac formula \eqref{eq:Kac} 
can be done by counting singular vectors at each level $M(c,h)_n$ 
constructed in Fact \ref{fct:singular}.
Here one should note that under the homomorphism
$M(c,h) \to \calF(\alpha)$ with the correspondence 
\eqref{eq:h-alpha}, we have
$$
 h=\dfrac{\left(\alpha_{r,s}^{\pm}\right)^2}{2}-\rho \alpha_{r,s}^\pm
  =\dfrac{\left(r t_{\pm}^2-s\right)^2-\left(t_{\pm}^2-1\right)^2}{8t_{\pm}^2}
  =h_{r,s}^\NS.
$$
We omit the detail of the proof 
and only address a technical but important issue:
the existence of the integration path $C$.
In the Virasoro case the corresponding integral is 
nothing but the Selberg integral \cite{S},
which will be reviewed in \S \ref{sect:Selberg}.

\subsection{Conjecture of Belavin-Bershtein-Tarnopolsky}

Let us denote by $\Lambda_{\bbC}$ the space of symmetric functions 
(see Notations in \S\ref{sect:intro}).
As is well-known, the bosonic Fock module is naturally identified with 
$\Lambda_{\bbC}$ by identifying the $n$-th bosonic generator $a_{-n}$ 
with the $n$-th power-sum symmetric functions.
In our situation we should also consider the fermionic part,
and it is natural to use the boson-fermion correspondence.
Following \cite{BBT}, let us use the following one.

\begin{lem}\label{lem:sf}
For a non-zero complex number $t \in \bbC$, 
the correspondence
\begin{align*}
&a_{n}  \longmapsto -2t n  \dfrac{\partial}{\partial p_{2 n}}, \quad
 a_{-n} \longmapsto -\dfrac{1}{2t} p_{2 n},\quad (n>0),\quad 
 b(z)   \longmapsto 
 \dfrac{1}{2\sqrt{2}} 
 \left(e^{\phi_-(z)} e^{2\phi_+(z)} - e^{-\phi_-(z)} e^{-2\phi_+(z)} \right)
\end{align*}
with
\begin{align*}
  \phi_{+}(z) := 
    \sum_{n \in \bbZ_{>0}}\dfrac{\partial}{\partial p_{2n-1}}z^{-n+1/2},\quad
  \phi_{-}(z) := 
   -\sum_{n \in \bbZ_{>0}}\dfrac{p_{2n-1}}{2n-1}z^{n-1/2}
\end{align*}
induces an isomorphism
$$
 \fs_t: \calF(\alpha) \longsimto \Lambda_{\bbC}
$$ 
of $\BF$-module with $\ket{\alpha} \longmapsto 1$.
\end{lem}

Here recall that the family  $\{p_\lambda\}$ of power-sum symmetric functions
is a basis of the space $\Lambda_{\bbQ}$ of symmetric functions.
 
\begin{proof}
We should check that the right hand side of the given correspondence 
satisfy the defining relation of $\BF$.
The relation of $a_n$'s is obviously satisfied.
For $b_n$'s, the defining relation \eqref{eq:HA:current} 
in terms of currents can be proved by 
$$
 \phi_{+}(z^2)\phi_{-}(w^2) \sim 
 \dfrac{1}{2}\log\dfrac{1-w/z}{1+w/z}
$$
and 
$$
  e^{\phi_{-}(z)} e^{2 \phi_{+}(z)} 
  e^{\phi_{-}(w)} e^{2 \phi_{+}(w)} 
 \sim 
  \dfrac{1-\sqrt{w/z}}{1+\sqrt{w/z}}
  \no e^{\phi_{-}(z)} e^{2 \phi_{+}(z)} 
  e^{\phi_{-}(w)} e^{2 \phi_{+}(w)} \no.
$$
\end{proof}

\begin{eg}\label{eg:chi:Lambda}
Applying $\fs_{t_+}$ to the singular vectors in Example \ref{eg:chi:Fock},
we have the following symmetric functions.
Here we suppress the symbol $+$ of $t_+$.
\begin{align*}
&\fs_t(\chi_{1,1}) = \dfrac{t-t^{-1}}{-2t} p_1,\\
&\fs_t(\chi_{3,1}) = 
  \dfrac{\left(t-t^{-1}\right)\left(3t-t^{-1}\right)}{-2t}
  \left(-\dfrac{2t^2}{3} p_3 - p_2 p_1 -\dfrac{t^2}{3} p_1^3\right),\\
&\fs_t(\chi_{1,3}) = 
  \dfrac{\left(t-t^{-1}\right)\left(t-3t^{-1}\right)}{-t}
  \left(\dfrac{1}{3}p_3-\dfrac{1}{2}p_2 p_1+\dfrac{1}{6}p_1^3\right).
\end{align*}
\end{eg}

The main result of this paper is to prove 

\begin{cnj}[{\cite{BBT}}]\label{cnj:main}
Let $t_+ := \rho+\sqrt{\rho^2+1}$ as \eqref{eq:tpm} in Fact \ref{fct:Kac}.
Under the isomorphism 
$\fs_{t_+}: \calF(\alpha^{+}_{r,s}) \to \Lambda_{\bbC}$,
the image of the singular vector $\chi^{+}_{r,s}$ satisfies
$$
 \fs_{t_+} (\chi^{+}_{r,s}) \propto P^{(1/t_{+}^2,2)}_{(r^s)},
$$
where the right hand side is the Uglov symmetric function
$P^{(\gamma,p)}_{\lambda}$ with prescribed parameters.
\end{cnj}

In the next \S\ref{subsec:Uglov} we review the Uglov symmetric functions.

\subsection{Uglov symmetric functions}
\label{subsec:Uglov}

Let $q$ and $t$  be generic complex numbers
and denote by $P_\lambda(q,t) \in \Lambda_{\bbC}$ 
the monic Macdonald symmetric function  
for a partition $\lambda$
with two parameters $q,t$
(see Notation in \S\ref{sect:intro} for the detail).

\begin{dfn}\label{dfn:Uglov}
For $p \in \bbZ_{>0}$, $\gamma \in \bbC$ and a partition $\lambda$,
set
$$
 P^{(\gamma,p)}_{\lambda}  := 
 \lim_{q \to 1} P_\lambda(\omega_p q, \omega_p q^{\gamma})
$$
and call it the Uglov symmetric function.
Here $\omega_p := \exp(2\pi\sqrt{-1}/p)$ is the $p$-th root of unity.
\end{dfn}

\begin{rmk}
This family of symmetric functions was introduced by D.~Uglov in \cite{U1,U2} 
under the name of Jack($\fgl_p$) symmetric function.
For $p=1$ it reduces to the usual Jack symmetric function.
\end{rmk}

\begin{eg}
We list $P_{(r^s)}(\alpha) := P^{(1/\alpha,2)}_{(r^s)}$ with $r s \le 4$.
\begin{align*}
&P_{(1)}(\alpha) = p_1,\quad
 -(1+\alpha)P_{(2)}(\alpha) = -p_2-\alpha p_1^2,\quad
 -2 P_{(1,1)}(\alpha) = p_2-p_1^2,\\
&-(1+\alpha) P_{(3)}(\alpha) 
 =-\dfrac{2\alpha}{3}p_3
  -p_2 p_1
  -\dfrac{\alpha}{3} p_1^3,\quad 
 P_{(1^3)}(\alpha) 
 =\dfrac{1}{3}p_3-\dfrac{1}{2}p_2p_1+\dfrac{1}{6}p_1^3,\\
&(1+\alpha)(1+3\alpha) P_{(4)}(\alpha) 
 =2\alpha p_4
 +\dfrac{8 \alpha^2}{3} p_3 p_1
 +p_2^2
 +2 \alpha p_2 p_1^2
 +\dfrac{\alpha^2}{3}p_1^4,
\\
&2(1+\alpha) P_{(2^2)}(\alpha) 
 =(\alpha-1)p_4
 -\dfrac{4 \alpha}{3}p_3 p_1
 +p_2^2
 +\dfrac{\alpha}{3}p_1^4,
\\
&8 P_{(1^4)}(\alpha) 
=-2 p_4+\dfrac{8}{3} p_3 p_1 + p_2^2 -2 p_2 p_1^2 +\dfrac{1}{3} p_1^4.
\end{align*}
Here we multiplicated appropriate factors so that each right hand side 
coincides with the integral form $J^{(2,\alpha)}_\lambda$ 
introduced in \cite{BBT}
\footnote{\cite{BBT} seems to have a typo at 
$J^{(2,\alpha)}_{\lambda}$ with $\lambda=(2^2)$.}.
One can now check Conjecture \ref{cnj:main} 
in Example \ref{eg:chi:Lambda}.
\end{eg}

\subsection{Check the case $r=1$}

Conjecture \ref{cnj:main} in the case $r=1$ is easy to show.

\begin{lem}
For any $s \in \bbZ_{>0}$, we have 
$P_{(1^s)}^{(\gamma,2)}=e_{s}$,
which is the $s$-th elementary symmetric function.
\end{lem}

\begin{proof}
At the Macdonald level we already have $P_{(1^s)}(q,t)=e_s$ 
\cite[Chap.\ VI, \S4, (4.8)]{M}.
\end{proof}

In this subsection let us suppress the symbol $\ep = -$ and 
denote $t := t_{-}$ and $\chi_{r,s} := \chi^{-}_{r,s}$ 
and so on for simplicity.
We have $\alpha_{r,s}=(r+1)t/2-(s+1)/2 t$.
In the case $r=1$ and $s$ odd,
the expression \eqref{eq:chi_rs} of the singular vector 
is mapped under $\fs_t$ to the symmetric function
\begin{align*}
\chi_{1,s}
&= \oint_C dz b(z) 
 \exp\left(2t \sum_{n\in \bbZ_{>0}}\dfrac{a_{-2n}}{2n}z^n\right) 
 e^{t \pi_0} z^{t a_0} \ket{\alpha_{1,s}-t} \\
&= \oint_C dz b(z) 
 \exp\left(2t \sum_{n\in \bbZ_{>0}}\dfrac{a_{-2n}}{2n}z^n\right) 
 z^{-(s+1)/2} \ket{\alpha_{1,s}} \\
&\longmapsto
 \dfrac{t}{2} \oint_C dz 
  \left(
  E_1(-z^{1/2})-E_1(z^{1/2})
  \right) 
  E_{0}(-z^{1/2})
  z^{-1-s/2}. 
\end{align*}
with
$$
 E_0(z) := 
 \exp\left(-\sum_{n \in \bbZ_{>0}} \dfrac{p_{2n}}{2n}z^{2n}\right),\quad 
 E_1(z) := 
 \exp\left(\sum_{n \in \bbZ_{>0}} \dfrac{p_{2n-1}}{2n-1}z^{2n-1}\right).
$$
By \cite[Chap.\ I, \S 2, (2.10), (2.14)]{M}
we have 
\begin{align*}
E_0(z) E_1(z) 
=\exp\left(-\sum_{n \in \bbZ_{>0}} \dfrac{p_{n}}{n}(-z)^n\right) 
= \sum_{m \in \bbZ_{\ge0}} z^m e_m.
\end{align*}
Therefore 
$$
 \left( E_1(-z^{1/2})-E_1(z^{1/2}) \right) 
 E_0(-z^{1/2})
=  \sum_{m \in \bbZ_{\ge0}} (-z^{1/2})^m e_m
  -\sum_{m \in \bbZ_{\ge0}} z^{m/2} e_m
=-2\sum_{n \in \bbZ_{>0}} z^{n-1/2} e_{2n-1}.
$$
Thus taking $C := \dfrac{1}{2\pi\sqrt{-1}}C_0$ 
with $C_0$ a closed circle around the origin $0 \in \bbC$,
the integration gives 
\begin{align*}
\fs_t(\chi_{1,s}) = -t e_{s}.
\end{align*}

\section{The proof}
\label{sect:main}

\subsection{Lemmas for Uglov symmetric functions}

\begin{lem}\label{lem:C}
\begin{enumerate}
\item 
Define
$$
C^0(z) :=
 \exp\left( 2\sum_{n>0} \dfrac{p_{2n-1}}{2n-1} z^{2n-1}\right)
 \exp\left(-2\sum_{n>0} \dfrac{\partial}{\partial p_{2n-1}} z^{-2n+1}\right)
$$
and set $C^0(z) = \sum_{n \in \bbZ} C^0_n z^{-n}$.
Then 
the Uglov symmetric function $P_\lambda^{(\gamma,2)}$ 
is an eigenfunction 
of the operator $C^0_0$. 
More precisely speaking, the following holds.
$$
 C^0_0 P_\lambda^{(\gamma,2)} = P_\lambda^{(\gamma,2)} \ve^0_\lambda,\quad
 \ve^0_\lambda := 1-2\sum_{i=1}^{\ell(\lambda)} (-1)^i \left((-1)^{\lambda_i}-1\right). 
$$

\item
Define
\begin{align}\label{eq:C10}
\begin{split}
C^1(\gamma;z) :=
&-\gamma \left(\sum_{n>0} p_{2n-1}z^{2n-1} \right)C^0(z) \\
&+C^0(z)
 \left( -\sum_{n>0}(2n-1) \dfrac{\partial}{\partial p_{2n-1}} z^{-2n+1}
        +\sum_{n>0} 2n  \dfrac{\partial}{\partial p_{2n}} z^{-2n}
        +\gamma \sum_{n>0} p_{2n} z^{2n} \right)
\end{split}
\end{align}
and set $C^1(\gamma;z) = \sum_{n \in \bbZ} C^1_n(\gamma) z^{-n}$.
Then 
$P_\lambda^{(\gamma,2)}$ 
is an eigenfunction of the operator $C^1_0(\gamma)$: 
$$
 C^1_0(\gamma) P_\lambda^{(\gamma,2)} 
  = P_\lambda^{(\gamma,2)} \ve^1_\lambda(\gamma),\quad
 \ve^1_\lambda(\gamma) := 
  -\sum_{i=1}^{\ell(\lambda)} (-1)^i \left\{
      2 (-1)^{\lambda_i} \lambda_i 
  + \gamma (1-2i) \left((-1)^{\lambda_i}-1\right) \right\}.
$$
\end{enumerate}
\end{lem}

\begin{proof}
We recall the result in \cite{SKAO} for the bosonization of the 
Macdonald difference operators and symmetric functions.
Put
\begin{align}\label{eq:eta}
\eta(z) := 
\exp\left(\sum_{n>0}(1-t^n)\dfrac{p_n}{n}z^n\right)
\exp\left(-\sum_{n>0}(1-q^n)\dfrac{\partial}{\partial p_n}z^{-n}\right)
\end{align}
and set $\eta(z) = \sum_{n \in \bbZ} \eta_n z^{-n}$.
Then the Macdonald symmetric function $P_\lambda(q,t) \in \Lambda_{\bbF}$ 
(see Notation in \S \ref{sect:intro} for the detail)
is the eigenfunction of the operator $\eta_0$,
and satisfies 
\begin{align*}
 \eta_0 P_\lambda(q,t) = P_\lambda(q,t) \ve_\lambda(q,t),\quad
 \ve_\lambda(q,t) := 1+ (t-1) \sum_{i=1}^{\ell(\lambda)} (q^{\lambda_i}-1)t^{-i}.
\end{align*}
Now take the limit $\hbar \to 0$ with $q=-e^{\hbar}$, $t=-e^{\gamma \hbar}$.
Then $P_\lambda(q,t) \to P^{(\gamma,2)}_\lambda$ and 
\begin{align}\label{eq:eta-hbar}
 \eta(z) = C^0(z) + \hbar C^1(\gamma;z) + O(\hbar^2),\quad
 \ve_\lambda(q,t) = \ve^0_\lambda + \hbar \ve^1_\lambda(\gamma) + O(\hbar^2),
\end{align}
where $C^0(z)$, $C^1(\gamma;z)$, $\ve^0_\lambda$ and $\ve^1_\lambda(\gamma)$ 
are given in the statement.
Now by the comparison of  the coefficients of $\hbar^0$ and $\hbar^1$,
the result follows immediately.
\end{proof}

\begin{prop}\label{prop:Uglov}
The Uglov symmetric function is characterized by the following properties.
\begin{itemize}
\item 
It has the expansion 
$P^{(\gamma,2)}_\lambda  = 
 m_{\lambda} + \sum_{\mu < \lambda} m_{\lambda,\mu} s_\mu$
with respect to the dominance ordering $<$ among partitions,
where 
$c_{\lambda,\mu} \in \bbC$. 

\item
It is the eigenfunction of the operator $C^1_0(\gamma)$ 
given in \eqref{eq:C10}. 
\end{itemize}
\end{prop}

\begin{proof}
This claim follows from the characterization of the 
Macdonald symmetric functions given in Notation in \S \ref{sect:intro}.
\end{proof}

\subsection{The main part of the proof}

Recall our isomorphism $\fs_t: \calF(\alpha) \to \Lambda_\bbC$ 
given in Lemma \ref{lem:sf}.
Using the same notation $\fs_t$ for $\End(\calF(\alpha)) \to \End(\Lambda_\bbC)$,
we have
\begin{align}\label{eq:fst-rec}
\fs_t\left( b(z^2) \right) 
= \dfrac{1}{2\sqrt{2}}
 \left(e^{\varphi_-(z)} e^{2\varphi_+(z)} - e^{-\varphi_-(z)} e^{-2\varphi_+(z)} \right),
\qquad
\fs_t\left( a(z^2)z \right)
= -\dfrac{1}{2t} \psi_{-}(z) -t \psi_{+}(z)
\end{align}
with
\begin{align*}
&\varphi_{+}(z) := 
   \sum_{n \in \bbZ_{>0}}\dfrac{\partial}{\partial p_{2n-1}}z^{-2n+1},\quad
 \varphi_{-}(z) := 
  -\sum_{n \in \bbZ_{>0}}\dfrac{p_{2n-1}}{2n-1}z^{2n-1},\\
&\psi_{+}(z) := 
  \sum_{n>0} 2n \dfrac{\partial}{\partial p_{2 n}} z^{-2n-1},\quad  
 \psi_{-}(z) := 
  \sum_{n>0} p_{2 n} z^{2n-1}.
\end{align*}

\begin{prop}\label{prop:C10}
The image $v_{r,s} := \fs_{t_+}(\chi_{r,s}^+) $ 
of the singular vector of the $\SVir$ module $\calF(\alpha_{r,s}^+)$
satisfies
$$
 C^1_0(t_+^{-2}) v_{r,s} = v_{r,s} \ve^1_{(r^s)}(t_+{-2}).
$$
\end{prop}

\begin{proof}
We deduce the result from the deformed Virasoro algebra case given in \cite{SKAO}.
Recall that the deformed Virasoro algebra is the associative algebra 
with two parameters $q,t \in \bbC$ generated by $T_n$'s ($n\in \bbZ$) over $\bbC$. 
It has the following bosonization map (injective homomorphism from the algebra 
to the completed enveloping algebra of the Heisenberg algebra):
\begin{align*}
\sum_{n \in \bbZ} T_n z^{-n} &=
\exp\left(
 -\sum_{n>0} \dfrac{1-t^n}{1+(q/t)^n}\dfrac{p_n}{n} \dfrac{z^n}{(q t)^{n/2}}\right)
\exp\left(
 -\sum_{n>0} (1-q^n)\dfrac{\partial}{\partial p_n} \dfrac{z^n}{(q/t)^{n/2}}\right) 
(q/t)^{1/2}q^{\alpha} \\
&+
\exp\left(
 \sum_{n>0} \dfrac{1-t^n}{1+(q/t)^n}\dfrac{p_n}{n} \dfrac{z^n}{(t^3/q)^{n/2}}\right)
\exp\left(
 \sum_{n>0} (1-q^n)\dfrac{\partial}{\partial p_n} \dfrac{z^n}{(t/q)^{n/2}}\right) 
(q/t)^{-1/2}q^{-\alpha}.
\end{align*}
Here the generators $T_n$ are considered as the operators acting on 
the space $\Lambda_\bbF$ of symmetric functions with coefficients with 
$\bbF = \bbQ(q,t)$.
The parameter $\alpha$ corresponds to the highest weight $h$ 
of this module $V_h$,
i.e., we have $V_h =  \bbF(T_{-1},T_{-2},\cdots)\ket{h}$ 
with $T_0 \ket{h}=h\ket{h}$.
The above map gives an isomorphism $V_h \simto \Lambda_\bbF$ 
with $\ket{h} \longmapsto 1$.

The generating series $\sum_n T_n z^{-n}$ 
is related to the operator $\eta(z)$ given at \eqref{eq:eta} 
in the following way.
Set 
\begin{align*}
\psi(z) := 
\exp\left(
-\sum_{n>0} \dfrac{1-t^n}{1+(q/t)^n}\dfrac{p_n}{n} \dfrac{z^n}{(t^3/q)^{n/2}}\right)
(q/t)^{-1/2}q^{-\alpha}.
\end{align*}
Then we have 
\begin{align*}
\psi(z) \left(\sum_{n \in \bbZ} T_n z^{-n}\right) = \eta(z/p^{1/2})
+\exp\left(
 \sum_{n>0} (1-q^n)\dfrac{\partial}{\partial p_n} \dfrac{z^n}{(t/q)^{n/2}}\right)  
(q/t)^{-1}q^{-2 \alpha}.
\end{align*}
Comparing the coefficients of $z^0$ in both sides, we have
\begin{align}\label{eq:pt-eta}
\sum_{n \ge 0} \psi_{-n} T_n = \eta_0 + (q/t)^{-1}q^{-2\alpha},
\end{align}
where we put $\psi(z) = \sum_{n \ge 0} \psi_{-n} z^n$.

The consequence is derived by taking the limit $\hbar \to 0$ with 
$q = -e^\hbar$ and $t=-e^{\hbar \gamma}$ of \eqref{eq:pt-eta}. 
In fact, we have the following expansion of $\sum_{n\in\bbZ}T_n z^{-n}$
in terms of $\hbar$.
\begin{align*}
&(q/t)^{-1/2}q^{-\alpha} 
\sum_{n \in \bbZ} T_n z^{-n} \\
=&\exp\left(\sum_{n>0}\dfrac{p_{2n-1}}{2n-1}z^{2n-1}\right) 
   \exp\left(-2\sum_{n>0}\dfrac{\partial}{\partial p_{2n-1}}z^{-2n+1}\right) 
   \left\{
    1+\hbar \left(\dfrac{\gamma}{2}\sum_{n>0}p_{2n} z^{2n} 
     +\sum_{n>0}\dfrac{\partial}{\partial p_{2n}} z^{-2n} \right) \right\} \\
  &+ \hbar(\gamma-2) 
     \left( \dfrac{1}{2}\sum_{n>0} \dfrac{p_{2n-1}}{2n-1} z^{2n-1}
           + \sum_{n>0} \dfrac{\partial}{\partial p_{2n-1}} z^{-2n+1}\right)\\
  &+\Biggl[\exp\left(-\sum_{n>0}\dfrac{p_{2n-1}}{2n-1}z^{2n-1}\right) 
   \exp\left(2\sum_{n>0}\dfrac{\partial}{\partial p_{2n-1}}z^{-2n+1}\right) 
   \left\{
    1-\hbar \left(\dfrac{\gamma}{2}\sum_{n>0}p_{2n} z^{2n} 
     +\sum_{m}\dfrac{\partial}{\partial p_{2n}} z^{-2n} \right) \right\} \\
  &\qquad + \hbar \gamma 
    \left( \dfrac{1}{2}\sum_{n>0} \dfrac{p_{2n-1}}{2n-1} z^{2n-1}
                     + \sum_{n>0} \dfrac{\partial}{\partial p_{2n-1}}z^{-2n+1}\right)
   \Biggr] \left(1-\hbar(1+2\alpha-\gamma)\right) + O(\hbar^2).
\end{align*}
Let us rewrite this expansion as 
\begin{align}
(q/t)^{-1/2}q^{-\alpha} 
\sum_{n \in \bbZ} T_n &z^{-n} 
=T^0(z) + \hbar T^1(\gamma;z) + O(\hbar^2).
\end{align}
Then obviously one has 
\begin{align*}
T^0(z) 
= 
&\exp\left(\sum_{n>0}\dfrac{p_{2n-1}}{2n-1}z^{2n-1}\right) 
   \exp\left(-2\sum_{n>0}\dfrac{\partial}{\partial p_{2n-1}}z^{-2n+1}\right) \\
&+\exp\left(-\sum_{n>0}\dfrac{p_{2n-1}}{2n-1}z^{2n-1}\right) 
   \exp\left(2\sum_{n>0}\dfrac{\partial}{\partial p_{2n-1}}z^{-2n+1}\right). 
\end{align*}

Similarly we have the expansion 
$$
 (q/t)^{1/2}q^{\alpha}  \psi(z)
 =\psi^0(z) + \hbar \psi^1(\gamma;z) + O(\hbar^2)
$$
of $\psi(z)$.
Thus taking the $\hbar^1$ part of \eqref{eq:pt-eta} 
and using \eqref{eq:eta-hbar},
one obtains
\begin{align}\label{eq:pt=C10}
 \sum_{n\ge0} \left(\psi^1_{-n}(\gamma) T^0_n + \psi^0_{-n}T^1_n(\gamma)\right)
= C^1_0(\gamma).
\end{align}
Here we used the $z$-expansions as 
$$
 \psi^0(z)=\sum_{n \ge 0}\psi^0_{-n} z^n,\quad
 \psi^1(\gamma;z)=\sum_{n \ge 0}\psi^1_{-n}(\gamma) z^n,\quad
 T^0(z)=\sum_{n \in \bbZ}T^0_{n} z^{-n},\quad
 T^1(\gamma;z)=\sum_{n \in \bbZ}T^1_{n}(\gamma) z^{-n}.
$$

The important point here is that when $\gamma=t^{-2}$ we have
\begin{align*}
T^1(t^{-2};z) = 
&\sqrt{2} z \fs_t\left(\ff\left(G(z^2)\right)-2\rho (\partial b)(z^2)\right) \\
&+(t^{-2}-1) 
  \left(\sum_{n>0} p_{2n-1}z^{2n-1}
      +2\sum_{n>0} (2n-1)\dfrac{\partial}{\partial p_{2n-1}}z^{-2n+1}\right) \\
&+(t^{-2}-1-2\alpha)
 \exp\left(-\sum_{n>0}\dfrac{p_{2n-1}}{2n-1}z^{2n-1}\right) 
   \exp\left(2\sum_{n>0}\dfrac{\partial}{\partial p_{2n-1}}z^{-2n+1}\right),
\end{align*}
which is the direct consequence of the map $\fs_t$ given at \eqref{eq:fst-rec} 
and the bosonization $\ff$ at \eqref{eq:bosonization}.
By this expression, we find
\begin{align}\label{eq:T1G}
 T^1_n(t^{-2}) \in \langle G_{1/2},G_{3/2},\cdots \rangle,
\end{align}
for $n >0$,
where the right hand side means the subalgebra of the enveloping algebra 
$U(\SVir)$ generated by the positive generators $G_{k}$'s ($k>0$).
Similarly one can prove 
\begin{align}\label{eq:T0G}
 T^0_n \in \langle G_{1/2},G_{3/2},\cdots \rangle,
\end{align}
for $n >0$,


Hence from \eqref{eq:pt=C10}, \eqref{eq:T1G},  \eqref{eq:T0G}
and the defining property
$$
 G_k \chi_{r,s}^+ = 0 \quad (k>0)
$$
of the singular vector, 
one can compute $C^1_0(t_{+}^{-2}) \fs_{t_+}(\chi_{r,s}^+)$ as
$$
 C^1_0(t_{+}^{-2}) \fs_{t_+}(\chi_{r,s}^+)= 
 \left(\psi^1_0(\gamma)T^0_0+\psi^0_0T^1_0(\gamma)\right)
 \fs_{t_+}(\chi_{r,s}^+).
$$
Now the consequence follows from the direct calculation 
using Lemma \ref{lem:C} (1).
\end{proof}

\begin{proof}[{Proof of Conjecture \ref{cnj:main}}]
By the expression \ref{eq:chi_rs} of the singular vector,
one can check that it has the expansion
$$
 \fs_{t^+}\left(\chi^+_{r,s}\right)
 \propto m_{(r^s)} + \sum_{\mu < (r^s)} c_{\mu} m_\mu
$$
with some coefficients $c_{\mu} \in \bbC$.
Now we see that Conjecture \ref{cnj:main} is true 
by Proposition \ref{prop:C10} and 
the characterization of the Uglov polynomial (Proposition \ref{prop:Uglov}).
\end{proof}

\appendix
\section{Selberg integral}
\label{sect:Selberg}

Let us follow  \cite[Chap.\ 8]{AAR},
which explains Aomoto's proof \cite{A} of the Selberg integral \cite{S}.

Let $n$ be a positive integer and $\alpha,\beta,\gamma \in \bbC$.
Set $C_n := [0,1]^n \subset \bbR^n$ and 
\begin{align*}
&w(x) = w_n(x;\alpha,\beta,\gamma) := 
 \prod_{i=1}^n x_i^{\alpha-1}(1-x_i)^{\beta-1}
 \prod_{1\le i < j \le n}|x_i-x_j|^{2 \gamma},
\\
&S_n(\alpha,\beta,\gamma) := \int_{C_n} dx w(x). 
\end{align*}
For a multi-index $m = (m_1,\ldots,m_n) \in \bbZ^n$ we also set 
\begin{align*}
S_n(m;\alpha,\beta,\gamma) := \int_{C_n} dx w(x) x^m.
\end{align*}

\begin{fct}\label{fct:SA}
If $\re \alpha > 0, \re \beta >0$ and 
$\re \gamma > -\min\{1/n,(\re \alpha)/(n-1),(\re \beta)/(n-1)\}$,
then
\begin{align*}
S_n(\alpha,\beta,\gamma) = \prod_{j=1}^n 
 \dfrac{\Gamma\left(\alpha+(j-1)\gamma\right) \Gamma\left(\beta+(j-1)\gamma\right)
        \Gamma\left(1+j \gamma\right)}
       {\Gamma\left(\alpha+\beta+(n+j-2)\gamma\right) \Gamma\left(1+\gamma\right)}
\end{align*}
and 
\begin{align*}
S_n\left((1^k);\alpha,\beta,\gamma\right) 
= \int_{C_n} dx w(x) \prod_{i=1}^k x_i 
= S_n(\alpha,\beta,\gamma) 
 \prod_{j=1}^n \dfrac{\alpha+(n-j)\gamma}{\alpha+\beta+(2n-j-1)\gamma}
\end{align*}
for $1\le k \le n$.
\end{fct}

The first integral is due to Selberg \cite{S} and 
the second is due to Aomoto \cite{A}.

For a while we use the abbreviation 
$S(k) := S_n\left((1^k);\alpha,\beta,\gamma\right)$.
Aomoto derived the integral $S(k)$ 
by constructing the following recursive formula among $S(k)$'s.

\begin{lem}\label{lem:AAR}
Setting $S(k) := S_n\left((1^k);\alpha,\beta,\gamma\right)$, we have
\begin{align*}
0 = \alpha S(k-1) - (\alpha+\beta) S(k-1) +\gamma (n-k) S(k-1)
   -\gamma (2n-k-1)S(k).
\end{align*}
\end{lem}

\begin{proof}
Consider the following integral of the total derivative.
\begin{align*}
 0&= \int_{C_n} dx \dfrac{\partial}{\partial x_1}
     \left[ w(x) (1-x_1) \prod_{i=1}^k x_i \right] \\
  &= \alpha \int_{C_n} dx w(x) (1-x_1) \prod_{i=2}^k x_i 
    -\beta \int_{C_n} dx w(x) \prod_{i=1}^k x_i 
    +2 \gamma \sum_{j=2}^n \int_{C_n} dx w(x) (1-x_1) \dfrac{\prod_{i=1}^k x_i }{x_1-x_j}.
\end{align*}
Then the result follows from 
\begin{align}\label{eq:8.2.1-1}
 \int_{C_n} dx w(x) \dfrac{\prod_{i=1}^k x_i }{x_1-x_j}
 =\begin{cases}
   0 & 2 \le j \le k \\
   \dfrac{1}{2}S(k-1) & k < j \le n
  \end{cases}
\end{align}
and
\begin{align}\label{eq:8.2.1-2}
 \int_{C_n} dx w(x) \dfrac{x_1 \prod_{i=1}^k x_i }{x_1-x_j}
 =\begin{cases}
   \dfrac{1}{2}S(k-1) & 2 \le j \le k \\
   S(k) & k < j \le n
  \end{cases}.
\end{align}
These equalities follows from the transposition $x_1 \longleftrightarrow x_j$.
For \eqref{eq:8.2.1-1} with $j \le k$,
it changes the sign of the integrand, so the integral vanishes.
In the case $j>k$, the same transposition leads to 
$$
 \dfrac{x_1}{x_1-x_j} \longmapsto 1-\dfrac{x_1}{x_1-x_j}
$$,
and the result immediately follows.
Similarly \eqref{eq:8.2.1-2} with $j \le k$ follows from
\begin{align*}
\dfrac{x_1^2x_j}{x_1-x_j} \longmapsto x_1 x_j-\dfrac{x_1^2x_j}{x_1-x_j}
\end{align*}
and the case $j > k$ follows from
\begin{align*}
\dfrac{x_1^2}{x_1-x_j} = x_1+\dfrac{x_1x_j}{x_1-x_j}.
\end{align*} 
\end{proof}

In the bosonization of Virasoro singular vector \cite{MY},
the following Selberg integral appears.
\begin{align*}
I(m)=I(m;r,t) := \int dz z^m \Phi(z),\quad
\Phi(z) = \Phi_r(z;t) := \prod_{i=1}^r z_i^{(1-r)t-1}\prod_{1\le i<j\le r}(z_i-z_j)^{2 t}.
\end{align*}
Here $t\in \bbC$, $r \in \bbZ_{\ge1}$ and $m \in \bbZ^r$.
In terms of $S_n(m;\alpha,\beta,\gamma)$ it is given by 
$$
 I(m) = S_r\left(m;(1-r)t,1,t\right).
$$
By Fact \ref{fct:SA} we have 
\begin{align*}
 I(0) = S_r((1-r)t,1,t) = \cdots = \prod_{j=1}^{r-1} 
   \dfrac{\Gamma\left((j-r)t\right) \Gamma\left(1+(j+1)t\right)}
         {\Gamma\left(1+t\right)}
\end{align*}
(this form appears in \cite{MY})
and
\begin{align*}
 I(1^k) =  S((1-r)t,1,t) \prod_{j=1}^k\dfrac{(1-j)t}{1+(r-j)t}=0
\end{align*}
for $1 \le k \le r$.

We now check
\begin{prop}
For $m \in \bbZ^r$ with $|m| := \sum_{i=1}^r m_i \neq 0$ 
we have $I(m) = 0$,
assuming the integral is defined.
\end{prop}

\begin{proof}
We can assume $m \in \bbZ_{\ge0}^r$.
Our proof is by the induction on $|m|$.
Similarly as in the proof of Lemma \ref{lem:AAR},
consider the integral 
\begin{align*}
 0 &= \int_{C_r} dx \dfrac{\partial}{\partial x_1}
      \left[(1-x_1)x^m w(x) \right] \\
   &= (\alpha-1+m_1)I(m-e_1)-(\alpha+\beta+m_1-1)I(m) 
      + 2\gamma \sum_{j=2}^r \int_{C_r} dx w(x) \dfrac{(1-x_1)x^m}{x_1-x_j}.
\end{align*}
Then the result comes from 
\begin{align}\label{eq:int0}
 \int_{C_r}dx w(x) \dfrac{x^m}{x_1-x_j} = 0
\end{align}
for any $m \in \bbZ_{\ge0}^r \setminus\{0\}$.
To show this equality,
consider the transposition $x_1 \longleftrightarrow x_j$. 
Then we have 
$$
 \dfrac{x^m}{x_1-x_j} \longmapsto 
 \dfrac{x^{s_{1 j}(m)}}{x_j-x_1} 
 =\dfrac{x^{s_{1 j}(m)}-x^{m}}{x_j-x_1} - \dfrac{x^m}{x_1-x_j}.
$$
Since the part $(x^{s_{1 j}(m)}-x^{m})/(x_j-x_1)$ is a polynomial 
and one can use the induction step,
we know that the integral \eqref{eq:int0} must vanish.
\end{proof}

\section{Taking the limit of numbers of variables in symmetric functions}

In this section we address an interpretation of the operators 
$C^0_0$ and $C^1_0(\gamma)$ introduced in \S \ref{sect:main}, Lemma \ref{lem:C}.
Let $x_1,x_2,\cdots,x_N$ be the indeterminates and consider 
the space 
$$
 \Lambda_N := \bbZ[x_1,\ldots,x_N]^{\mathfrak{S}_N}
$$
of symmetric polynomials.
The base change of $\Lambda_N$ is denoted as before:
$\Lambda_{N,R} := \Lambda_{N} \otimes_{\bbZ} R$ 
for a ring $R$.
We have a projective system
$\{(\Lambda_{N,R})_{N\in\bbZ}, (p_{M,N})_{M \ge N}\}$ 
of $R$-modules, where 
$$
 \pr_{M,N}: \Lambda_{M,R} \longto \Lambda_{N,R},\quad
 x_i \longmapsto 0\ (i>N)
$$ 
is the restriction map.
The space $\Lambda$ of symmetric functions is nothing but 
the projective limit of this system.
In particular we have a natural projection 
$$
 \pr_N: \Lambda_R \to \Lambda_{N,R}
$$ 
for each $N$
(see \cite[Chap.\ I]{M} for the detail).

\begin{dfn}
We say that an operator $O$ acting on $\Lambda_R$ is 
the limit of the operators $O_{(N)}$ acting on $\Lambda_{N,R}$,
and denote $O = \varprojlim O_{(N)}$ if $O$ and $O_{(N)}$'s are 
compatible with the maps $\pr_{M,N}$ and $\pr_N$.
\end{dfn}

Here we recall the technique of \cite{AMOS} for 
the calculation of the limit of differential  operators.

\begin{lem}
Set a formal series
$$
 A(x) := \sum_{n>0} x^n \dfrac{\partial}{\partial p_n}.
$$
Then the projection map $\pr_N:\Lambda_{\bbQ} \to \Lambda_{N,\bbQ}$ 
is expressed as
$$
 \pr_N = \exp\left(\sum_{i=1}^{N} A(x_i) \right).
$$
\end{lem}

The proof is immediate so we omit it.

As a corollary, 
if $O_{(N)}$'s are compatible with $\pr_{M,N}$'s and satisfies 
\begin{align}\label{eq:var-lim}
 O_{(N)} \exp\left(\sum_{i=1}^{N} A(x_i) \right) = 
  \exp\left(\sum_{i=1}^{N} A(x_i) \right) O
\end{align}
then we have $O = \varprojlim O_{(N)}$.

\begin{prop}
We denote by
$$
 D_i := x_i \dfrac{\partial}{\partial x_i},\quad
 T_{q,i} := (x_i \longmapsto q x_i)
$$
the Euler operator and the $q$-shift operator for the variable $x_i$ 
respectively.
\begin{enumerate}
\item 
Consider the operator
\begin{align*}
 C^0_{(N)} := 2 (-1)^{N-1}
  \sum_{i=1}^N 
  \prod_{j \neq i}\left( -\dfrac{x_i+x_j}{x_i-x_j}\right)T_{-1,i}.
\end{align*}
Then we have 
$$
 C^0_0 = \varprojlim C^0_{(N)}.
$$

\item
The difference-differential operator
\begin{align*}
 C^1_{(N)} := \dfrac{1}{2} (-1)^{N-1}
  \sum_{i=1}^N \prod_{j \neq i}
   \left( -\dfrac{x_i+x_j}{x_i-x_j}\right)
   \left(D_{i}+\gamma \sum_{k \neq i}\dfrac{x_i}{x_i+x_k}\right)
   T_{-1,i}
\end{align*}
satisfies
$$
 C^1_0(\gamma) = \varprojlim C^1_{(N)}(\gamma).
$$
\end{enumerate}
\end{prop}

\begin{proof}
The proof is by checking \eqref{eq:var-lim} directly and 
we only write down the proof for (1).
The proof of (2) is similar and we omit it.

Using 
\begin{align*}
T_{-1,i}\exp\left(\sum_{i=1}^N A(x_i)\right)
= \exp\left(\sum_{i=1}^N A(x_i)\right) \exp\left(A(-x_i)-A(x_i)\right),
\end{align*}
the right hand side of \eqref{eq:var-lim} for $O_{(N)}=C^0_{(N)}$ 
can be calculated  as
\begin{align}\label{eq:C0-lim:1}
C^0_{(N)}\exp\left(\sum_{i=1}^N A(x_i)\right)
=2(-1)^{N-1}
 \exp\left(\sum_{i=1}^N A(x_i)\right)
 \sum_{i=1}^N \prod_{j \neq i} \left(-\dfrac{x_i+x_j}{x_i-x_j}\right)
 \exp\left(-2\sum_{n>0}x_i^{2n-1} \dfrac{\partial}{\partial p_{2n-1}}\right).
\end{align}

Now recall the symmetric function $g_n(q,t)= Q_{(n)}(q,t) \propto P_{(n)}(q,t)$ 
given in \cite[Chap.\ VI]{M}.
It can be defined as the expansion
\begin{align*}
\sum_{n\ge0}g_n(q,t) z^n = 
 \exp\left(\sum_{n>0}\dfrac{1-t^n}{1-q^n}\dfrac{p_n}{n}z^n\right)
\end{align*}
Using $\pr_N(p_r)=\sum_{i=1}^N x_i^r$, one can show
$$
\pr_N\left(g_r(0,t)\right) = 
\begin{cases}
 (1-t)\sum_{i=1}^N x_i^r 
 \prod_{j \neq i} \left(-\dfrac{x_i+x_j}{x_i-x_j}\right)
 & r \le 1 \\
 1 & r=0
\end{cases}
$$
In particular, if $t=-1$, then we have 
\begin{align*}
 \sum_{n\ge0}g_n(0,-1) z^n = 
 \exp\left(2\sum_{n>0}\dfrac{p_{2n-1}}{2n-1}z^{2n-1}\right),\quad
 \pr_N\left(g_r(0,-1)\right) = 
 2\sum_{i=1}^N x_i^r 
 \prod_{j \neq i} \left(-\dfrac{x_i+x_j}{x_i-x_j}\right)
\end{align*}
for $r>0$.

Let us define $g^*_n(q,t)$ by the generating series
\begin{align*}
\sum_{n\ge0}g^*_n(q,t) z^n = 
 \exp\left(\sum_{n>0}\dfrac{1-t^n}{1-q^n}\dfrac{\partial}{\partial p_n}z^n\right).
\end{align*}
Then the computation of \eqref{eq:C0-lim:1} can be continued as
\begin{align*}
&2(-1)^{N-1}
 \sum_{i=1}^N \prod_{j \neq i} \left(-\dfrac{x_i+x_j}{x_i-x_j}\right)
 \exp\left(-2\sum_{n>0}x_i^{2n-1} \dfrac{\partial}{\partial p_{2n-1}}\right) \\
&=2(-1)^{N-1}
  \sum_{i=1}^N \prod_{j \neq i} \left(-\dfrac{x_i+x_j}{x_i-x_j}\right)
  \sum_{r\ge0} (-x_i)^r g^*_r(0,-1) \\
&=\sum_{r\ge0} (-1)^r  g^*_r(0,-1) \pr_N\left(g_r(0,-1)\right) 
 = \pr_N C^0_0.
\end{align*}
Thus (1) is proved.
\end{proof}


\end{document}